\documentclass[12pt,reqno]{amsart}
\usepackage{amsmath,amssymb,amsfonts,amscd,latexsym,amsthm,mathrsfs}
\usepackage{color}
\usepackage{relsize}
\usepackage{hyperref}
\usepackage{fullpage}

\newtheorem{theorem}{Theorem}[section]
\newtheorem{lemma}[theorem]{Lemma}
\newtheorem{corollary}[theorem]{Corollary}
\newtheorem{proposition}[theorem]{Proposition}
\theoremstyle{remark}
\newtheorem{remark}{Remark}

\renewcommand{\d}{{\mathrm d}}

\newcommand{\ord}{\operatorname{ord}}

\newcommand{\Z}{\mathbb Z}

\providecommand{\alaplus}{\genfrac{}{}{0pt}{}{}{+}} 
\providecommand{\aladots}{\genfrac{}{}{0pt}{}{}{\cdots}} 

\providecommand{\doK}[1]{\vcenter{#1\kern.2ex\hbox{\normalfont\text{K}}\kern.2ex}}

\title{Euler's factorial series, Hardy integral, and continued fractions}

\author{Anne-Maria Ernvall-Hyt\"onen}
\address{Anne-Maria Ernvall-Hyt\"onen, Mathematics and Statistics, P.O. Box 68, 00014 University of Helsinki, Finland}
\email{anne-maria.ernvall-hytonen@helsinki.fi}

\author{Tapani Matala-aho}
\address{Tapani Matala-aho, Aalto University, Department of Mathematics, P.O. Box 11100, FI-00076 Aalto, Finland }
\email{tapani.matala-aho@aalto.fi}

\author{Louna Sepp\"al\"a}
\address{Louna Sepp\"al\"a, Mathematics and Statistics, P.O. Box 68, 00014 University of Helsinki, Finland}
\email{louna.seppala@helsinki.fi}

\begin{document}

\begin{abstract}
We study $p$-adic Euler's series $E_p(t) = \sum_{k=0}^{\infty}k!t^k$ at a point $p^a$, $a \in \Z_{\ge 1}$, and use Pad\'e approximations 
to prove a lower bound for the $p$-adic absolute value of the expression $cE_p\left(\pm p^a\right)-d$, where $c, d \in \Z$. 
It is interesting that the same Pad\'e polynomials which $p$-adically converge to $E_p(t)$, approach the Hardy integral 
$\mathcal{H}(t) = \int_{0}^{\infty} \frac{e^{-s}}{1-ts}\d s$ on the Archimedean side. 
This connection is used with a trick of analytic continuation when deducing an Archimedean bound for the numerator Pad\'e polynomial 
needed in the derivation of the lower bound 
for $|cE_p\left(\pm p^a\right)-d|_p$.
Furthermore, we present an interconnection between $E(t)$ and $\mathcal{H}(t)$ via continued fractions. 
\end{abstract}

\maketitle

\section{Introduction}

\subsection{Background}

The series
\begin{equation}\label{Eulerfactorial}
E(t) := \sum_{k=0}^{\infty}k!t^k,
\end{equation}
named after Euler, converges in the standard $p$-adic metric given by $|p|_p=p^{-1}$ for a prime $p$, when $t$ belongs to the disc 
$\left\{z\in\mathbb{C}_p \; \left| \; |z|_p < p^{\frac{1}{p-1}} \right. \right\}$. 
Hence a function $E_p$ is defined on that disc via the $p$-adic values $E_p(z):=\sum_{n=0}^\infty n!z^n$.
The $p$-adic completion of the rationals $\mathbb{Q}$ with respect to the metric $|\cdot|_p$ is denoted by $\mathbb{Q}_p$,
while $\mathbb{C}_p$ denotes the completion of the algebraic closure of $\mathbb{Q}_p$.
The notation $\overline{\mathbb{C}}:=\mathbb{C}\cup \{\infty\}$ is used for the one point compactification of $\mathbb{C}$.

The integral
\begin{equation*}
\mathcal{H}(t) = \int_{0}^{\infty} \frac{e^{-s}}{1-ts}\d s
\end{equation*}
is a companion to the series \eqref{Eulerfactorial} in the sense that its Taylor series expansion at the origin is $E(t)$.
Even though the integral $\mathcal{H}(t)$ is not defined on the positive real axis, it can be analytically continued there: there exists a global multivalued function, say $\widehat{\mathcal{H}}(t)$, which is analytic on the Riemann surface of the logarithm function and which on $\overline{\mathbb{C}}\setminus[0,\infty]$ coincides with $\mathcal{H}(t)$ (see Section \ref{Analyticcontinuation}).
Note that $\mathcal{H}(t)$ is also defined at the origin but is not analytic there because 
the radius of convergence of its Taylor series expansion $E(t)$ is zero.

Not much is known about the nature of the values of the
$p$-adic Euler's series $E_p(t) = \sum_{k=0}^{\infty}k!t^k$ at non-zero rational points $t\in\mathbb{Q}^*$ for any prime $p$.
Are they rational or irrational, and if irrational, are they algebraic or transcendental? It is a folklore expectation that 
both $E_p(\pm 1) = \sum_{k=0}^{\infty}k!(\pm 1)^k$ are transcendental, although even their irrationality has not been proved yet.
Of course, if we knew the $p$-adic expansion of some number to be ultimately periodic, then that number would be rational.
However, the periodicity of the $p$-adic expansion of, say, $E_p(-1) = \sum_{k=0}^{\infty} k!(-1)^k$ is an open question for all primes. 
The $p$-adic expansion is given by $E_p(-1) = \sum_{k=0}^{\infty} a_k p^k$, where the coefficients $a_k=a_k(p)\in\{0,1,\ldots,p-1\}$
are called the digits in base $p$. 
For example, the first 15 digits $a_0a_1\ldots a_{14}$ of $E_2(-1)$ and $E_3(-1)$ are $001111010011001$ and $221000110102102$, respectively. 

Irrationality and transcendence questions are ultimately related to the theory of Diophantine approximations
which focuses on lower and upper bounds of linear forms with algebraic coefficients.
In this work, our aim is to derive $p$-adic lower bounds for the linear form $cE_p\left(\pm p^a\right)-d$ with integer coefficients $c,d$,
where $a$ is a positive integer.
To give a flavour of our results (outlined in more detail in Section \ref{sec:newresults}): Let $p$ be a prime number and  $H\in\mathbb{Z}_{\ge 4}$. 
Suppose that $a$ is a positive integer such that
$
p^a > c_1 \log \left( c_2 H \right),
$
where $c_1,c_2$ are constants given in Corollary \ref{Cormin}.
Then, for all $c,d \in \Z$, $c \neq 0$ with $|c|+|d|\le H$, there holds
\begin{equation}\label{lowerboundintro}
\left| cE_p\left(\pm p^a\right)-d \right|_p > \left( 2H e^{\frac{11}{16}} \right)^{ - \frac{32}{11} a \log p}.
\end{equation}
In particular, we see that $E_p\left(\pm p^a\right)$ is not a rational number $d/c$ satisfying $|c|+|d|\le H$.

There exists an explicit version of the value $\mathcal{H}(-1)$ related to $E_p(-1) = \sum_{k=0}^{\infty}k!(-1)^k$: 
already G.\ H.\ Hardy \cite[formula 2.4.6]{Hardy1973} proved that
\begin{equation}\label{}
\mathcal{H}(-1) =  - e \left( \gamma  + \sum_{n=1}^{\infty}  \frac{(-1)^n}{n\cdot n!} \right) = 0.5963473623\ldots,
\end{equation}
where the right-hand side is a real number known as the Euler-Gompertz constant and often denoted by $\delta$.
The irrationality of the Euler-Gompertz constant $\delta$ as well as that of Euler's gamma constant 
$\gamma=0.5772156649\ldots$ is an open question. 
As noted by Rivoal \cite{Rivoal2012}, it follows from Mahler's work \cite{Mahler1968} that at least one of the constants $\gamma$ and $\delta$ is transcendental.
Rivoal also presented another proof to this fact and showed that the pair $(\gamma, \delta)$ cannot be very well approximated with a pair of rational numbers having the same denominator.
Various other mathematical developments related to Euler's gamma $\gamma$ and the Euler-Gompertz constant $\delta$ are thoroughly explored by J.\ Lagarias in his extensive article \cite{Lagarias2013}. 
 
The correspondence between the functions $E(t)$ and $\mathcal{H}(t)$ runs even deeper, as indicated by Matala-aho and Zudilin \cite{TAWA2018}. 
Namely, both $E(t)$ and $\mathcal{H}(t)$ (more generally $\widehat{\mathcal{H}}(t)$) share the common Pad\'e polynomials 
(see Section \ref{padechapter}).
On the $p$-adic side, these shared Pad\'e approximations converge towards $E(t)$, and on the Archimedean side towards $\mathcal{H}(t)$ 
with appropriate values of $t$. 
A realisation of this is a nice interconnection between $E(t)$ and $\mathcal{H}(t)$ via the continued fraction
\begin{equation}\label{contfrac}
\frac{1}{1-t-\frac{t^2}{1-3t-\frac{2^2t^2}{1-5t+\ldots} } } := \frac{a_1}{b_1 + \frac{a_2}{b_2 + \frac{a_3}{b_3 +\ldots} } },
\end{equation}
where $a_1=1$, $b_1=1-t$, $a_k=(k-1)^2t^2$ and $b_k=1-(2k-1)t$ for $k\in\mathbb{Z}_{\ge 2}$.

\begin{theorem}\label{CFpadic+real}
Let $p\in\mathbb{P}$ and $t\in\mathbb{C}_p$, $|t|_p<1$. Then, in the field of $p$-adic numbers,
\begin{equation}\label{pcontinued}
\frac{1}{1-t-\frac{t^2}{1-3t-\frac{2^2t^2}{1-5t+\ldots} } } = \sum_{k=0}^{\infty}k!t^k.
\end{equation}
Let $t\in\mathbb{R}$, $t\le 0$. Then, in the field of real numbers,
\begin{equation}\label{rcontinued}
\frac{1}{1-t-\frac{t^2}{1-3t-\frac{2^2t^2}{1-5t+\ldots} } } = \int_{0}^{\infty} \frac{e^{-s}}{1-ts}\d s.
\end{equation}
\end{theorem}

For the basics of generalized continued fractions and the proof of Theorem \ref{CFpadic+real}, see 
Section \ref{Generalized continued fractions}.

\begin{remark}\label{pcremark}
We shall prove that the continued fraction on the left-hand side of equation \eqref{pcontinued} converges in the disk
$\left\{t\in\mathbb{C}_p \; \left| \; |t|_p < 1 \right. \right\}$. 
As mentioned above, the Euler series on the right-hand side of equation \eqref{pcontinued} converges in the larger disk
$\left\{t\in\mathbb{C}_p \; \left| \; |t|_p < p^{\frac{1}{p-1}} \right. \right\}$.
Not much is known about the convergence of the continued fraction in \eqref{pcontinued} on the unit circle 
$\left\{t\in\mathbb{C}_p \; \left| \; |t|_p = 1 \right. \right\}$. 
In the case $t=-1$, there is an infinite sequence of rational numbers $p_n/q_n$ converging 
$p$-adically towards $E_p(-1) = \sum_{k=0}^{\infty}k!(-1)^k$ and in Archimedean metrics towards
$\mathcal{H}(-1)$ (see \cite{TAWA2018}). The sequence $(p_n/q_n)$ is related to a subsequence of the convergents 
of the continued fraction \eqref{pcontinued} via the above mentioned Pad\'e approximations.
\end{remark}

\subsection{Objectives}

Inspired by Salah-Eddine Remmal \cite{Remmal980}, we study $p$-adic lower bounds for linear forms
\begin{equation}\label{linform}
\Lambda_p := c E_p\left(\pm p^a\right)-d, \quad c,d \in \Z,\quad c \neq 0,
\end{equation}
in the $p$-adic values $E_p(\pm p^a)$ of Euler's series
$$
E(t) = \sum_{n=0}^\infty n!t^n
$$
at a point $t=\pm p^a$, where $p$ is a prime and $a\in\mathbb{Z}_{\ge 1}$. 

Our target is two-fold: we shall revise Remmal's lower bound result, but also introduce a method for making the bounds as sharp as possible using Pad\'e approximations. 
We shall apply Pad\'e approximations to the series $E(t)$ generated by the identity 
$Q_{l}(t)E(t) - P_{l}(t) = R_{l}(t)$, $l\in \mathbb{Z}_{\ge 1}$; see \eqref{PADE3}.
In the Archimedean metrics, the methods producing such lower bounds use upper bounds of the denominator polynomials $Q_{l}(t)$ only.
However, on the $p$-adic side, we need Archimedean bounds for the numerator polynomials $P_{l}(t)$, too.
A good estimate for $Q_{l}(t)$ is easily gained by a contour integral;  
to get a reasonably good bound for $P_{l}(t)$, we apply the Archimedean `Pad\'e identity'  
$Q_l(t) \mathcal{H}(t) - P_l(t) = \mathcal{R}_l(t)$
for $\mathcal{H}(t)$. So, we need estimates for $\mathcal{H}(t)$ and the remainder $\mathcal{R}_l(t)$.
In the case $t<0$, this is straightforward. The case $t>0$ is a bit more challenging because $\mathcal{H}(t)$ is not even defined then. 
The method of analytic continuation will do the trick here.
In Section \ref{Chapterestimates}, it is shown that
$Q_l(t) \mathcal{H}^{\tau}(t) - P_l(t) = \mathcal{R}_l^{\tau}(t)$,
where $\tau :=\{s=y(1+bi)\; | \; y:0\to \infty\}$, $b\in\mathbb{R}^+$, and 
\begin{equation*}
\mathcal{H}^{\tau}(t) = \int_{\tau} \frac{e^{-s}}{1-ts}\d s,\quad
\mathcal{R}_l^{\tau}(t) =  l! t^{2l} \int_{\tau} \frac{e^{-s}s^l}{(1-ts)^{l+1}} \d s   
\end{equation*} 
are analytic continuations of $\mathcal{H}(t)$ and $\mathcal{R}_l(t)$, respectively.
Now $\mathcal{H}^{\tau}(t)$ and $\mathcal{R}_l^{\tau}(t)$ are defined on the positive real axis and may be estimated by their integral representations. Then we optimize the bounds with respect to $b$. It seems that this method is new and hopefully can be applied also elsewhere.

\subsection{Earlier Diophantine results}\label{Diophantine}

Let $\mathbb{K}$ be any number field and let $F_v(t)$ denote the value of a series $F(t)\in\mathbb{K}[[t]]$ at a point $t$ in the $v$-adic domain $\mathbb{K}_v$.
Let $P\in \mathbb{K}[x_1, \ldots, x_m]$ be a polynomial in $m$ variables and suppose 
$F_1(t),\ldots, F_m(t)\in \mathbb{K}[[t]]$ are power series. 
Take a $\xi\in\mathbb{K}$. A relation $P(F_1(\xi), \ldots, F_m(\xi)) = 0$ is called global if it holds in all the fields $\mathbb{K}_v$, where all the series $F_1(\xi), \ldots, F_m(\xi)$ converge.
Let $\mathbb{Z}_{\mathbb{K}}$ denote the ring of integers in $\mathbb{K}$ and let
$V_0$ and $V_{\infty}$ denote the set of non-Archimedean and Archimedean valuations, respectively.

Euler series is a member of the class of $F$-series introduced by Chirski\u\i\ in \cite{Chirski1989}.
In the same and subsequent works, Chirski\u\i\ also answered the problem of the existence of global relations between members of the 
class of $F$-functions. 
Later Bertrand, Chirski\u\i, and Yebbou \cite{bertrandetal2004} refined the results by effectively estimating the prime $p$ for which there exists a valuation $v|p$ breaking the global relation:
Suppose $h_i\in\mathbb{Z}_{\mathbb{K}}$ are not all zero and $F_0(t) \equiv 1, F_1(t), \ldots, F_m(t)$ are $F$-series that are 
linearly independent over $\mathbb{K}(z)$ and constitute a solution to a differential system $D$, and $\xi\in\mathbb{K}\setminus\{0\}$ 
is an ordinary point of the system $D$. Write
$$
\Lambda_m(t):=h_0F_0(t) + h_1F_1(t) + \ldots + h_mF_m(t)
$$
and assume that
\begin{equation*}
\prod_{v \in V_\infty} \max_{0 \le i \le m} \left\{ \| h_i \|_v \right\} \le H.
\end{equation*}
Then there exists an infinite collection of intervals each containing a prime number $p$ such that for some valuation $v|p$ it holds
\begin{equation}\label{BCYbound}
\left\| \Lambda_m(\xi) \right\|_{v} > H^{-(m+1) - c(m)/\sqrt{\log\log H} },
\end{equation}
where $c(m)$ is a positive constant depending on $m$. (For the definition of $F$-series and their interplay with polyadic numbers, see the discussion in \cite{Chirski2019}.)

Write
$F_k(t):=E(\alpha_k t)$ for $k=1,\ldots,m$, where $\alpha_1, \ldots, \alpha_m \in \mathbb{Z}_{\mathbb{K}} \setminus \{0\}$ are
$m \ge 1$ pairwise distinct, non-zero algebraic integers.
Then $F_k(t)$, $k=1,\ldots,m$, constitute a system of $F$-functions satisfying the above assumptions. Consequently, the linear form
$
h_0 + h_1 E_v(\alpha_1) + \ldots + h_m E_v(\alpha_m)$ 
satisfies bound \eqref{BCYbound}. 
This result was improved and made totally explicit by Sepp\"al\"a \cite{Seppala2019}.   
By assuming $\log H \ge se^s$,
where $s$ is explicitly given in terms of $m$ , $\alpha_1, \ldots, \alpha_m$ and $\kappa = [\mathbb{K} : \mathbb{Q}]$, 
she showed that there exists a prime
\begin{equation*}\label{interval}
p \in \left] \log \left( \frac{\log H}{\log \log H} \right), \frac{17m\log H}{\log \log H} \right[
\end{equation*}
and a valuation $v|p$ for which
\begin{equation*}\label{LSbound}
\left\| h_0 + h_1 E_v(\alpha_1) + \ldots + h_m E_v(\alpha_m) \right\|_{v} > H^{-(m+1) - 114 m^2 
\cdot  \frac{\log \log \log H}{\log \log H}}.
\end{equation*}

Matala-aho and Zudilin \cite{TAWA2018} in turn proved a completely different result which was generalized by 
Ernvall-Hyt\"onen, Matala-aho, and Sepp\"al\"a \cite{AMLOTA2019} as follows:

\begin{proposition}\label{AB} \cite[Theorem 3]{AMLOTA2019} 
Let $\xi \in \mathbb{Z}\setminus\{0\}$ and suppose $h_0,h_1\in \mathbb{Z}$ are not both zero.
Assume $T\subseteq\mathbb{P}$ is a subset of primes such that the set $T\setminus S$ satisfies the condition 
\begin{equation*}\label{riistajaehto}
\limsup_{n\to\infty} c^nn!\prod_{p\in T\setminus S} |n!|_p^2=0,
\ \text{where}\ c=c(\xi;R):=4|\xi|\prod_{p\in R}|\xi|_p^2
\end{equation*}
for any finite subset $S$ of $T$.
Then there exist infinitely many primes $p\in T$ such that
$$
h_0 + h_1 E_p(\xi) \ne 0.
$$
\end{proposition}

Let $h\in\mathbb{Z}_{\ge 3}$ be a given integer. Then Proposition \ref{AB} holds when the set $T$ is chosen to be the union of the primes 
in $r$ residue classes in the reduced residue system modulo $h$, where $r > \frac{\varphi(h)}{2}$. 
This in turn has the following generalization proved in \cite[Theorem 10.2]{Seppala2019}:
Let $m\in\mathbb{Z}_{\ge 1}$ and $V=\{ v \in V_0 \; : \; v|p \ \text{for some} \ p \in T \}$. 
Suppose $r > \frac{m \varphi (h)}{m+1}$. Then there exists a valuation $v \in V$ such that
$
h_0 + h_1 E_v(\alpha_1) + \ldots + h_m E_v(\alpha_m) \neq 0.
$

In the above mentioned works, only the existence of a prime (in a finite or infinite set) breaking a given global relation is proved.
This phenomenon seems to be common in almost all the works considering the arithmetic of the values of Euler's series.
 
Suppose now that the prime, say $p$, is fixed. Remmal \cite{Remmal980} (see also \cite{Remmal981}) considered lower bounds for the quantities  
$\left| S(E(p^a)) \right|_{p}$, where $S(z)\in\mathbb{Z}[z]\setminus\{0\}$, assuming $a\in\mathbb{Z}_{\ge 1}$ is big enough with respect to $p$ and $H$, the naive height of the polynomial $S(z)$.
In the case $S(z)=cz-d$, Remmal's result is given by:

\begin{proposition}\cite[Th\'eor\`eme, pp.\ 3--5]{Remmal980}
Let $p$ be a prime number and $H\in\mathbb{Z}_{\ge 1}$. 
Suppose that $a$ is a positive integer such that
\begin{equation}\label{paehto5}
p^a > \;  (\log H)^{64},\quad \log H\ge 2^{13}.
\end{equation}
Then for all $c,d \in \Z$, $c \neq 0$ with $|c|+|d|\le H$, there holds
\begin{equation}\label{Remmalbound}
\left|c E_p(p^a)-d\right|_p \ge  \frac{1}{ H^{\frac{a}{8}\log p} }.
\end{equation}
\end{proposition}

As far as the authors know, Remmal's result is sole of its kind.
However, the lower bound of $p^a$ in \eqref{paehto5}  is superfluously large, and our target is to revise it.

\subsection{New results}\label{sec:newresults}

We shall prove the following general theorem, together with some corollaries. 

\begin{theorem}\label{theoremX}
Let $p$ be a prime number and $H,l\in\mathbb{Z}_{\ge 1}$. 
Let $t\in\mathbb{Z} \setminus \{0\}$, $|t|_p<1$ and denote
$$
B(l,t): = \left( \frac{l|t|}{4} \right)^{\frac{1}{4}}+\left( \frac{l|t|}{4} \right)^{-\frac{1}{4}}.
$$
If
\begin{equation}\label{paehto2}
2 (l+1)! B(l+1,t) \cdot |t|^{l+1} e^{2\sqrt{\frac{l+1}{|t|}}} \cdot |t|_p^{2l} < \frac{1}{H},
\end{equation}
then for all $c,d \in \Z$, $c \neq 0$ with $|c|+|d|\le H$, there holds
\begin{equation*}
\left| c E_p\!\left(t\right)-d \right|_p \ge  |t|_p^{2l}.
\end{equation*}
\end{theorem}

\begin{corollary}\label{theorem1}
Let $p$ be a prime number and $H,l\in\mathbb{Z}_{\ge 1}$. 
Suppose that $a$ is a positive integer such that
\begin{equation}\label{paehto}
p^a > \left( H \cdot 2 (l+1)! B(l+1,p^a) e^{2\sqrt{\frac{l+1}{p^a}}} \right)^{\frac{1}{l-1}},\quad 
B(l,\pm p^a)= \left( \frac{lp^a}{4} \right)^{\frac{1}{4}}+\left( \frac{lp^a}{4} \right)^{-\frac{1}{4}}.
\end{equation}
Then for all $c,d \in \Z$, $c \neq 0$ with $|c|+|d|\le H$, there holds
\begin{equation}\label{plowerbound}
\left| c E_p\left(\pm p^a\right)-d \right|_p \ge \frac{1}{p^{2al}}.
\end{equation}
\end{corollary}

Now we are ready to choose appropriate values for $l$ as a function of $H$.
For example, in order to get Remmal's lower bound \eqref{Remmalbound} we pick $l=\left\lfloor \frac{\log H}{16} \right\rfloor$.

\begin{corollary}\label{Cordiminish}
Let $p$ be a prime number, $c,d \in \Z$, $c\neq 0$, $H\in\mathbb{Z}_{\ge 1}$ and $H(c,d):= |c|+|d|\le H$.
Suppose that $a$ is a positive integer such that
\begin{equation}\label{paehto6}
p^a > (\log H)^4,\quad \log H\ge e^{8}.
\end{equation}
Then
\begin{equation*}\label{}
\left|c E_p\left(\pm p^a\right)-d\right|_p > \frac{1}{ H^{\frac{a}{8}\log p} }.
\end{equation*}
\end{corollary}

Note that our assumptions \eqref{paehto6} give a larger scope for the size of $p^a$ in comparison to the size of $\log H$ 
than assumptions \eqref{paehto5} used by Remmal. 
However, we can offer even better scope by choosing $l=\left\lceil \frac{16}{11}\log(2H) \right\rceil$. In fact, this choice will give 
an approximate minimum of the right-hand side of inequality \eqref{paehto}. 
In this case, the price we need to pay is that the bound for $\left|c E_p\left(\pm p^a\right)-d \right|_p$ becomes weaker.

\begin{corollary}\label{Cormin}
Let $p$ be a prime number and  $H\in\mathbb{Z}_{\ge 4}$. 
Suppose that $a$ is a positive integer such that
\begin{equation}\label{paehto4}
p^a > \frac{16}{11} e^{\frac{11 + 6 \log 4  + \log 5 + 4 \sqrt{10}}{11} } \log \left(2H e^{\frac{11}{16}} \right).
\end{equation}
Then for all $c,d \in \Z$, $c \neq 0$ with $|c|+|d|\le H$, there holds
\begin{equation}\label{Measure4}
\left| c E_p\left(\pm p^a\right)-d \right|_p > \frac{1}{\left( 2H e^{\frac{11}{16}} \right)^{\frac{32}{11} a \log p}}.
\end{equation}
\end{corollary}

Note that by Theorem \ref{CFpadic+real}, all the above results for the $p$-adic values of Euler's series at $t=\pm p^a$ are 
valid for the $p$-adic values of the continued fraction in \eqref{contfrac} at $t=\pm p^a$, too.

\section{Analytic continuation of the Hardy integral}\label{Analyticcontinuation}

For $0\le\alpha<\frac{\pi}{2}$, define a circular path (an arc) 
$\kappa_{R,\alpha}:= \left\{ \left. Re^{i\beta} \; \right| \; 0\le \beta \le \alpha \right\}$ and
line paths (segments) $\tau_{\alpha} := \left\{ \left. re^{i\alpha}\; \right| \; 0\le r\le\infty \right\}$ and 
$\tau_{R,\alpha} := \left\{ \left. re^{i\alpha}\; \right| \; 0\le r\le R \right\}$
with their inverse paths denoted by $\tau^{-1}_{\alpha}$ and $\tau^{-1}_{R,\alpha}$, respectively. 
Now we are ready to construct an integration loop by setting
$\gamma_{R,\alpha}:=\tau_{R,0}\cup \kappa_{R,\alpha}\cup \tau^{-1}_{R,\alpha}$.
Then, by Cauchy's integral theorem,
\begin{equation*}
\oint_{\gamma_{R,\alpha}} \frac{e^{-s}}{1-ts}\d s = \theta_{\alpha}\frac{e^{-1/t}}{t},
\end{equation*}
where $\theta_{\alpha}=-2\pi i$ if $\frac{1}{t}$ is within the integration path and $0$ otherwise
for any $t\in\mathbb{C}\setminus\{0\}$ with $\frac{1}{t} \notin\gamma_{R,\alpha}$. Consequently,
\begin{equation}\label{cauchybits}
\int_{\tau_{R,0}} \frac{e^{-s}}{1-ts}\d s  + \int_{\kappa_{R,\alpha}} \frac{e^{-s}}{1-ts}\d s  +  
\int_{\tau^{-1}_{R,\alpha}} \frac{e^{-s}}{1-ts}\d s  = \theta_{\alpha}\frac{e^{-1/t}}{t}.
\end{equation}
Write $t=Te^{i\mu}$. Assuming $TR>1$, we estimate 
\begin{equation*}
\left| \int_{\kappa_{R,\alpha}} \frac{e^{-s}}{1-ts}\d s \right| 
= \left| \mathlarger\int_0^\alpha \frac{e^{-R\cos\beta}e^{-iR\sin\beta}}{1-TRe^{i(\beta+\mu)}} iR e^{i\beta}\d \beta \right|
\le  R\alpha \cdot \frac{e^{-R\cos\alpha}}{TR-1}.					
\end{equation*}

If $-\frac{\pi}{2}<\alpha\le 0$, we may proceed in a similar manner. Therefore, as $\cos\alpha>0$ and $R\rightarrow\infty$, 
we see from \eqref{cauchybits} that
\begin{equation*}
\int_{\tau_{\alpha}} \frac{e^{-s}}{1-ts}\d s  +\theta_{\alpha}\frac{e^{-1/t}}{t} = \int_{\tau_{0}} \frac{e^{-s}}{1-ts}\d s =  \int_{0}^{\infty} \frac{e^{-s}}{1-ts}\d s
\end{equation*}
for $t\in\overline{\mathbb{C}}\setminus\left(\tau_{0}\cup\tau_{-\alpha}\right)$.
Now we denote  
\begin{equation*}
\mathcal{H}_{\alpha}(t) := \int_{\tau_{-\alpha}} \frac{e^{-s}}{1-ts}\d s, 
\end{equation*}
which is defined and analytic for all $t\in\overline{\mathbb{C}}\setminus\tau_{\alpha}$. 
In particular,
\begin{equation*}\label{Hintegral}
\mathcal{H}(t)=\mathcal{H}_{0}(t) = \int_{0}^{\infty} \frac{e^{-s}}{1-ts}\d s
\end{equation*}
is analytic for all $t\in\overline{\mathbb{C}}\setminus[0,\infty]$.

Taking, for example, $\mathcal{H}_{\pi/4}(t)$, we get an analytic continuation of $\mathcal{H}(t)$ to the positive real axis.
In this manner we create a global function , say $\widehat{\mathcal{H}}(t)$, consisting of the branches $\mathcal{H}_{\alpha}(t)$
defined on the branched sheets $e^{2\pi i k}\overline{\mathbb{C}}\setminus\{\tau_{\alpha}\}$, $k\in \mathbb{Z}$. 
In a nutshell, $\widehat{\mathcal{H}}(t)$ is an analytic multivalued function on the Riemann surface 
$\bigcup_{k\in \mathbb{Z}} \left( e^{2\pi i k}\overline{\mathbb{C}}\setminus\{0\} \right)$ of the logarithm 
and which on $\overline{\mathbb{C}}\setminus\{\tau_{\alpha}\}$ coincides with $\mathcal{H}_{\alpha}(t)$.

There exists an explicit version of $\widehat{\mathcal{H}}(t)$: already G.\ H.\ Hardy \cite[formula 2.4.6]{Hardy1973} 
computed the integral $\mathcal{H}(t)$.

\begin{proposition}\cite{Hardy1973}
\begin{equation}\label{Hardyexplicit}
\mathcal{H}(t) = \frac{e^{-\frac{1}{t}}}{t} \left( \gamma + \log\left(-\frac{1}{t}\right) + 
\sum_{n=1}^{\infty}  \frac{\left(\frac{1}{t}\right)^n}{n\cdot n!} \right)
\end{equation}
for $t\in\overline{\mathbb{C}}\setminus[0,\infty[$. Moreover, the right-hand side of \eqref{Hardyexplicit} gives
the analytic continuation of $\mathcal{H}(t)$ to the positive real axis $(0,\infty)$.
\end{proposition}

Because analytic continuations are unique, we get:

\begin{corollary}
\begin{equation*}\label{GlobalHardyexplicit}
\widehat{\mathcal{H}}(t) = \frac{e^{-\frac{1}{t}}}{t} \left( \gamma + \log\left(-\frac{1}{t}\right) + 
\sum_{n=1}^{\infty}  \frac{\left(\frac{1}{t}\right)^n}{n\cdot n!} \right)
\end{equation*}
for $t\in \bigcup_{k\in \mathbb{Z}} \left( e^{2\pi i k}\overline{\mathbb{C}}\setminus\{0\} \right)$.
\end{corollary}

For example, 
\begin{equation*}\label{}
\widehat{\mathcal{H}}(-1) =  - e \left( \gamma + 2\pi i k + \sum_{n=1}^{\infty}  \frac{(-1)^n}{n\cdot n!} \right),\quad k\in \mathbb{Z}
\end{equation*}
\begin{equation*}\label{}
\widehat{\mathcal{H}}(1) =  \frac{1}{e} \left( \gamma +  i\pi +  2\pi i k + \sum_{n=1}^{\infty}  \frac{1}{n\cdot n!} \right),\quad k\in \mathbb{Z},
\end{equation*}
and in particular,
\begin{equation*}\label{}
\mathcal{H}(-1) =  - e \left( \gamma  + \sum_{n=1}^{\infty}  \frac{(-1)^n}{n\cdot n!} \right) = 0.5963473623\ldots,
\end{equation*}
where the right-hand side is the Euler-Gompertz constant $\delta$.

\section{Pad\'e approximations}\label{padechapter}

From Matala-aho and Zudilin's work \cite{TAWA2018}, we get the following Pad\'e approximation formulae:

\begin{lemma}\label{EULERPADE}
We have
\begin{equation}\label{PADE3}
Q_{l}(t)E(t) - P_{l}(t)=R_{l}(t),
\end{equation}
where
\begin{equation*}\label{seriesexp1}
Q_l(t)=\sum_{h=0}^{l}h!\binom{l}{h}^2(-t)^h,\quad P_l(t)=[Q_l(t)E(t)]_{l-1},
\end{equation*}
\begin{equation}\label{Remainderseries}
R_l(t)=l!^2 t^{2l}\sum_{k=0}^{\infty} k!\binom{l+k}{k}^2t^k,
\end{equation}
by using the shorthand notation 
$\left[\sum_{k=0}^{\infty}f_kt^k\right]_m=\sum_{k=0}^{m}f_kt^k.$
\end{lemma}

We shall consider the Hardy integral
\begin{equation*}
\mathcal{H}(z)=\int_{0}^{\infty} \frac{e^{-s}}{1-zs}\d s.
\end{equation*}

By denoting $D=\frac{\d}{\d t}$, we define 
\begin{equation*}
L_n(t)=e^t D^n\left(e^{-t}t^n\right),
\end{equation*}
\begin{equation*}
\mathcal{Q}_n(z)=(-z)^n L_n\!\left(\frac{1}{z}\right),
\end{equation*}
\begin{equation*}
\mathcal{P}_n(z)= (-1)^nz^{n-1} \mathlarger\int_{0}^{\infty} \frac{\left(L_n(s)-L_n\left(\frac{1}{z}\right)\right)e^{-s}}{s-1/z}\d s,
\end{equation*}
and
\begin{equation*}
\mathcal{R}_n(z)= n!z^{2n}  \int_{0}^{\infty}   \frac{s^ne^{-s}}{(1-zs)^{n+1}}\d s 
\end{equation*}
for $ n\in\mathbb{Z}_{\ge 0}$.

\begin{lemma}\label{HARDYAPPROX}
Then we have
\begin{align*}
\mathcal{Q}_n(z)  &= Q_n(z)=(-z)^n L_n\!\left(\frac{1}{z}\right),\\
\mathcal{Q}_n(z) & \mathcal{H}(z) -\mathcal{P}_n(z)=\mathcal{R}_n(z),\\
 \mathcal{P}_n(z) &= P_n(z)
\end{align*}
for all $n\in\mathbb{Z}_{\ge 0}$.
\end{lemma}

\begin{proof}
By Leibniz's rule we get
\begin{align*}
L_n(t)&=e^t \sum_{k=0}^{n}\binom{n}{k}D^ke^{-t}D^{n-k}t^n\nonumber\\
      &=\sum_{k=0}^{n}\binom{n}{k}(-1)^k n(n-1)\cdots (k+1) t^k\nonumber \\
      &=\sum_{k=0}^{n}(n-k)!\binom{n}{k}^2(-t)^{k}.
\end{align*}
This verifies $\mathcal{Q}_n(z) =(-z)^n L_n(1/z)= Q_n(z)$. Then we turn to the second claim by evaluating
\begin{equation}\label{secondclaim}
\begin{split}
\mathcal{Q}_n(z) & \mathcal{H}(z) -\mathcal{P}_n(z) \\
     &= (-z)^n L_n(1/z) \int_{0}^{\infty} \frac{e^{-s}}{1-zs}\d s + (-1)^n z^{n} \mathlarger\int_{0}^{\infty} \frac{\left(L_n(s)-L_n\left(\frac{1}{z}\right)\right)e^{-s}}{1-zs}\d s\\
     &=(-z)^n \int_{0}^{\infty} \frac{L_n(s)e^{-s}}{1-zs}\d s \\
     &=(-z)^n \int_{0}^{\infty} \frac{D^n(e^{-s}s^n)}{1-zs}\d s,\\ 
		 &=(-z)^n \left( \Big|_{0}^{\infty} \frac{D^{n-1}(e^{-s}s^n)}{1-zs}  - 
		   z \int_{0}^{\infty} \frac{D^{n-1}(e^{-s}s^n)}{(1-zs)^2}\d s \right),\\
		 &= \ldots\\ 
\end{split}
\end{equation}
where we apply partial integration repeatedly, and use the observation  that
\begin{equation*}\label{}
D^k\left(e^{-s} s^n \right) = e^{-s}V_k(s),\quad V_k(s)\in \mathbb{Z}[s],\quad 
\underset{s=0}{\ord}\,  V_k(s)\ge 1,
\end{equation*}
for all $k=0,1,\ldots,n-1$, if $n\ge 1$.

Therefore,
\begin{equation}\label{Hardyapproximation}
\mathcal{Q}_n(z) \mathcal{H}(z) - \mathcal{P}_n(z) = n! z^{2n} \int_{0}^{\infty} \frac{e^{-s}s^n}{(1-zs)^{n+1}}\d s = \mathcal{R}_n(z). 
\end{equation}
Further, we see that $\mathcal{P}_n(z) \in \mathbb{Q}[z]$ and $\deg \mathcal{P}_n(z) \le n-1$. Next the integral 
\begin{equation*}\label{}
\int_{0}^{\infty} \frac{e^{-s}s^n}{(1-zs)^{n+1}}\d s  
\end{equation*}
may be expanded to a formal series (see \eqref{Remainderseries}), showing that $\underset{z=0}{\ord}\, \mathcal{R}_n(z) = 2n$.
Finally, the uniqueness of Pad\'e approximations shows that $\mathcal{P}_n(z) = P_n(z)$.
\end{proof}

\section{Estimates}\label{Chapterestimates}

Let us estimate the term $R_n(z)$  in the $p$-adic case.

\begin{lemma}\label{jaannoskoko}
Let $t\in\mathbb{C}_p$, $|t|_p\le 1$. Then
\begin{equation*}\label{}
|R_n(t)|_p \le |l!|_p^2\, |t|_p^{2l}.
\end{equation*}
In particular, 
\begin{equation*}\label{}
|R_n(\pm p^a)|_p = |l!|_p^2\, p^{-2la}
\end{equation*}
for $a\in\mathbb{Z}_{\ge 1}$.
\end{lemma}

\begin{proof}
We have
\begin{equation*}\label{}
\left| R_n\left(t\right) \right|_p  = \left| l!^2 t^{2l} \right|_p \left| \sum_{k=0}^{\infty} k! \binom{l+k}{k}^2 t^k \right|_p 
\le |l!|_p^2\, |t|_p^{2l} \max_k \left\{ \left| k! \binom{l+k}{k}^2 t^k \right|_p \right\}
\le |l!|_p^2\, |t|_p^{2l}. 
\end{equation*}
This proves the claim.
\end{proof}

For the polynomial $Q_n(t)$ we get the following bound:

\begin{lemma}\label{polrajaQ}
Let $t\in \mathbb{R} \setminus \{0\}$. Then
\begin{equation*}\label{}
|Q_n(t)| \leq n!|t|^n e^{2\sqrt{\frac{n}{|t|}}}.
\end{equation*}
\end{lemma}

\begin{proof}
First we estimate the term $L_n(t)$.
Let $\gamma=\left\{ \left. z=t+re^{i\alpha} \; \right| \; \alpha:0\to 2\pi,\ r>0\right\}$ be a path around $t$. Then
\begin{align*}
L_n(t)&=e^tD^n\left(e^{-t}t^n\right)\nonumber\\
&=e^t \frac{n!}{2\pi i}  \oint_{\gamma} \frac{e^{-z}z^n}{(z-t)^{n+1}}\d z\nonumber \\
&=\frac{n!}{2\pi i}  \int_{0}^{2\pi} \frac{e^{-re^{i\alpha}} (re^{i\alpha} +t)^n}{(re^{i\alpha})^{n+1}} re^{i\alpha}i\d\alpha. \nonumber
\end{align*}
Thus we get the bound
\begin{equation*}
|L_n(t)| \le n! \cdot \frac{e^{r} (r+|t|)^n}{r^{n}},
\end{equation*}
valid for all $r>0$. We may now further estimate this expression. We have
\[
n! \cdot \frac{e^{r} (r+|t|)^n}{r^{n}}
= n! e^{r} \left(1+\frac{|t|}{r}\right)^n
= n! e^{r} \left(\left(1+\frac{1}{\frac{r}{|t|}}\right)^{\frac{r}{|t|}}\right)^{\frac{n|t|}{r}}
\leq n!e^{r+\frac{n|t|}{r}}.
\]
The function $r+\frac{n|t|}{r}$ obtains its minimum when $r=\sqrt{n|t|}$, and this minimum is $2\sqrt{n|t|}$. Hence, we have proved
\begin{equation*}\label{}
|L_n(t)|\leq n!e^{2\sqrt{n|t|}}.
\end{equation*}
The bound for $|Q_n(t)|$ follows from $Q_n(t)=(-t)^n L_n\left(\frac{1}{t}\right)$. 
\end{proof}

\begin{lemma}\label{PPOLRAJA1}
Let $t\in \mathbb{R}$, $t<0$. Then
\begin{equation*}\label{}
|P_n(t)| \leq 2 n!|t|^n e^{2\sqrt{\frac{n}{|t|}}}.
\end{equation*}
\end{lemma}

\begin{proof}
From  Lemma \ref{HARDYAPPROX}  we get $P_n(t) = Q_n(t) \mathcal{H}(t) - \mathcal{R}_n(t)$.
Therefore, we need estimates for $|\mathcal{H}(t)|$ and $|\mathcal{R}_n(t)|$. Immediately, 
\begin{equation*}
|\mathcal{H}(t)| = \int_{0}^{\infty} \frac{e^{-s}}{1-ts}\d s \le \int_{0}^{\infty} e^{-s} \d s =1, 
\end{equation*}
for all $t\le 0$. Next (following \cite{TAWA2018}),
\begin{equation}\label{remainderestimate}
\begin{split}
|\mathcal{R}_n(t)| &=	n! |t|^{2n} \int_{0}^{\infty} \frac{e^{-s}s^n}{(1-ts)^{n+1}}\d s \\
&\le n! |t|^{2n} \max_{s\ge 0} \left\{\frac{s^n}{(1-ts)^{n+1}} \right\} \int_{0}^{\infty} e^{-s} \d s \\
&= n! |t|^{n} \cdot \frac{n^n}{(n+1)^{n+1}}. 
\end{split}
\end{equation}
Thereby
\begin{align*}
|P_n(t)| &\le  |Q_n(t)|\,|\mathcal{H}(t)| + |\mathcal{R}_n(t)| \\
&\le n! |t|^{n} \left(e^{2\sqrt{\frac{n}{|t|}}} + \frac{n^n}{(n+1)^{n+1}}\right)\\
&\le n! |t|^{n} e^{2\sqrt{\frac{n}{|t|}}}\left(1 + \frac{n^n}{(n+1)^{n+1}e^{2\sqrt{\frac{n}{|t|}}}}\right)\\
&\le 2 n! |t|^{n} e^{2\sqrt{\frac{n}{|t|}}}.\\
\end{align*}
\end{proof}

The above proof was based on the integral representation
\begin{equation*}
\mathcal{H}(t) = \int_{0}^{\infty} \frac{e^{-s}}{1-ts}\d s, 
\end{equation*}
which does not converge if $t\in \mathbb{R}^+$.
In the case $t>0$ we use analytic continuation, and for that, we need the gamma function.

Recall that the gamma function $\Gamma(x)$ is defined by
\begin{equation*}
\Gamma(x+1) = \int_{0}^{\infty} e^{-s}s^x\d s 
\end{equation*}
for $x\in \mathbb{R}^+$. It satisfies the functional equation $\Gamma(x+1) = x \Gamma(x)$ for $x\in \mathbb{R}^+$.

\begin{lemma}
Let $b\in \mathbb{R}$ and choose $\tau :=\{s=y(1+bi)\; | \; y:0\to \infty\}$.
Then
\begin{equation}\label{gammaproperty}
\Gamma(n+1) = \int_{\tau} e^{-s}s^n\d s 
\end{equation}
for all $n\in\mathbb{Z}_{\ge 0}$.
\end{lemma}

\begin{proof}
The proof is a standard application of partial integration. Let first $n=0$. Then
\begin{equation*}
\int_{\tau} e^{-s}\d s = \int_0^\infty e^{-y(1+bi)} (1+bi) \d y = \bigg/_{\!\!\!\!\!\!0}^\infty -e^{-y(1+bi)} = 1=\Gamma(1). 
\end{equation*}
Suppose then that \eqref{gammaproperty} is true for any $n=k\ge 0$. Now
\begin{align*}
\int_{\tau} e^{-s}s^{k+1}\d s &= \int_0^\infty e^{-y(1+bi)} (y(1+bi))^{k+1} (1+bi) \d y\\
&= \bigg/_{\!\!\!\!\!\!0}^\infty -e^{-y(1+bi)} (y(1+bi))^{k+1} + \int_0^\infty (k+1) e^{-y(1+bi)} (y(1+bi))^k (1+bi)\d y\\
&= 0 + (k+1) \int_\tau e^{-s} s^k \d s\\
& = (k+1)\Gamma(k+1) \\
&= \Gamma(k+2).
\end{align*}
\end{proof}

\begin{lemma}\label{Ppolraja2}
Let $t\in \mathbb{R}$, $t>0$. Then
\begin{equation*}\label{}
|P_n(t)| \leq 2 n! B t^{n}  e^{2\sqrt{\frac{n}{t}}},\quad B= B(n,t)=\left(\frac{nt}{4} \right)^{\frac{1}{4}}+\left( \frac{nt}{4} \right)^{-\frac{1}{4}}.
\end{equation*}
\end{lemma}

\begin{proof}
Let $b \in \mathbb{R}^+$ and $\tau :=\{s=y(1+bi)\; | \; y:0\to \infty\}$. Note that 
$\tau = \tau_{\beta} = \left\{ \left. re^{i\beta}\; \right| \; r:0\to \infty,\ \beta= \arctan b \right\}$, where $0<\beta<\frac{\pi}{2}$.
Then  
\begin{equation*}
\mathcal{H}_{-\beta}(t) := \int_{\tau_{\beta}} \frac{e^{-s}}{1-ts}\d s 
\end{equation*}
is defined for all $t \ge 0$. Set
\begin{equation*}\label{}
\mathcal{P}_n^{\tau}(t) := -(-1)^n t^{n} \mathlarger\int_{\tau} \frac{\left(L_n(s)-L_n\left(\frac{1}{t}\right)\right)e^{-s}}{1-ts}\d s. 
\end{equation*}
Following \eqref{secondclaim}, we deduce
\begin{align*}
Q_n(t) & \mathcal{H}_{-\beta}(t)  -\mathcal{P}_n^{\tau}(t) \\
&= (-t)^n L_n\!\left(\frac{1}{t}\right) \int_{\tau} \frac{e^{-s}}{1-ts}\d s + (-1)^n t^{n} \mathlarger\int_{\tau} \frac{\left(L_n(s)-L_n\left(\frac{1}{t}\right)\right)e^{-s}}{1-ts}\d s\\
&=(-t)^n  \int_{\tau} \frac{L_n(s)e^{-s}}{1-ts}\d s \\
&=(-t)^n  \int_{\tau} \frac{D^n\left(e^{-s}s^n\right)}{1-ts}\d s. 
\end{align*}
Where again we apply partial integration repeatedly  to obtain
\begin{equation*}\label{}
Q_n(t) \mathcal{H}_{-\beta}(t)  - \mathcal{P}_n^{\tau}(t) = n! t^{2n} \int_{\tau} \frac{e^{-s}s^n}{(1-ts)^{n+1}}\d s := \mathcal{R}_n^{\tau}(t). 
\end{equation*}

Next we show that
\begin{equation*}\label{}
\mathcal{P}_n^{\tau}(t) = P_n(t). 
\end{equation*}
Let us  denote $L_n(s)=\sum_{k=0}^{n} l_ks^k$. Then 
\begin{align*}
\mathcal{P}_n^{\tau}(t) & = (-1)^n t^{n-1} \mathlarger\int_{\tau} \frac{\left(L_n(s)-L_n\left(\frac{1}{t}\right)\right)e^{-s}}{s-\frac{1}{t}}\d s\\
& = (-1)^n t^{n-1} \mathlarger\int_{\tau} 	\sum_{k=0}^{n} l_k \left(s^{k-1}+\frac{s^{k-2}}{t} + \frac{s^{k-3}}{t^2} + \ldots + \frac{1}{t^{k-1}} \right) e^{-s} \d s\\
& = (-1)^n t^{n-1} \sum_{k=0}^{n} l_k \left(\int_{\tau} s^{k-1}e^{-s}\d s  + \frac{1}{t} \int_{\tau} s^{k-2}e^{-s}\d s + \ldots +\frac{1}{t^{k-1}} \int_{\tau} e^{-s}\d s \right)\\
& = (-1)^n t^{n-1} \sum_{k=0}^{n} l_k \left(\Gamma(k) + \frac{\Gamma(k-1)}{t} + \ldots + \frac{\Gamma(1)}{t^{k-1}} \right) = P_n(t).
\end{align*}

Now we are ready to proceed as in the proof of Lemma \ref{PPOLRAJA1}.
First we have
\begin{equation}\label{Ftauestimate}
\begin{split}
|\mathcal{H}_{-\beta}(t)| &\le \int_{0}^{\infty} \left| \frac{e^{-y}e^{-iby}}{1-ty-ibty}(1+bi)\right|\d y \\
&= \sqrt{1+b^2} \mathlarger\int_{0}^{\infty} \frac{e^{-y}}{\sqrt{(ty-1)^2+(bty)^2}}\d y \\
&= b+\frac{1}{b},
\end{split}				
\end{equation}
where $(ty-1)^2+(bty)^2$ attains its minimum  $\frac{b^2}{1+b^2}$ at $y=\frac{1}{t(1+b^2)}$.

Next we bound the remainder term in a similar manner:
\begin{equation*}\label{}
\begin{split}
|\mathcal{R}_n^{\tau}(t)| &=  n! t^{2n} \left| \int_{\tau} \frac{e^{-s}s^n}{(1-ts)^{n+1}} \d s \right|\\
&\le n! t^{2n} \left(\sqrt{1+b^2}\right)^{n+1} \mathlarger{\mathlarger\int_{0}^{\infty}} \frac{e^{-y}y^n}{\left(\sqrt{(ty-1)^2+(bty)^2}\right)^{n+1}} \d y  \\
&\le n! t^{2n} \left(\sqrt{1+b^2}\right)^{n+1} \max_{y \in [0,\infty[} \left\{ \frac{y^n}{\left(\sqrt{(ty-1)^2+(bty)^2}\right)^{n+1}} \right\} \int_{0}^{\infty} e^{-y} \d y.
\end{split}				
\end{equation*}
Write
\begin{equation*}\label{}
I(y):= \frac{y^n}{\left(\sqrt{(ty-1)^2+(bty)^2}\right)^{n+1}}
= \frac{1}{\left(\sqrt{(t-\frac{1}{y})^2+(bt)^2}\right)^{n}} \cdot \frac{1}{\sqrt{(ty-1)^2+(bty)^2}}.
\end{equation*}
Then
\begin{equation}\label{melkeinmax}
I(y)\le \frac{1}{\left(\sqrt{(bt)^2}\right)^{n}} \cdot \frac{1}{\sqrt{\frac{b^2}{1+b^2}}} =
\frac{\sqrt{1+b^2}}{t^{n}b^{n+1}},\quad 0 \le y < \infty,
\end{equation}
where the estimate in the second term is coming from in \eqref{Ftauestimate}.
But 
\begin{equation*}\label{}
I\!\left(\frac{1}{t}\right) = \frac{1}{t^{n}b^{n+1}},
\end{equation*}
which shows that we can not improve \eqref{melkeinmax} essentially.
Consequently
\begin{equation*}\label{}
|\mathcal{R}_n^{\tau}(t)| \le n! t^{n} \left(b+\frac{1}{b}\right) \left(\sqrt{1+\frac{1}{b^2}}\right)^{n}. 
\end{equation*}

Finally
\begin{equation*}\label{}
|P_n(t)| \le  |Q_n(t)|\,|\mathcal{H}_{-\beta}(t)| + |\mathcal{R}_n^{\tau}(t)|
\le n! t^{n} \left(b+\frac{1}{b}\right)\left(e^{2\sqrt{\frac{n}{t}}} +\left(\sqrt{1+\frac{1}{b^2}}\right)^{n}\right).
\end{equation*}
Put now $b= \left(\frac{nt}{4} \right)^{\frac{1}{4}}$. Then
\begin{equation*}\label{}
|P_n(t)| \le 2 n! t^{n}  e^{2\sqrt{\frac{n}{t}}} \left( \left(\frac{nt}{4} \right)^{\frac{1}{4}} + \left(\frac{nt}{4} \right)^{-\frac{1}{4}}\right).
\end{equation*}
\end{proof}

\section{Recurrences}

\begin{lemma}\label{POLREC}
The Pad\'e polynomials $P_l(t)$ and  $Q_l(t)$ satisfy the recurrences
\begin{equation}\label{POLRECURRENCES}
\begin{split}
P_l(t) & = (1-(2l-1)t) P_{l-1}(t) - (l-1)^2t^2 P_{l-2}(t),\\
Q_l(t) & = (1-(2l-1)t) Q_{l-1}(t) - (l-1)^2t^2 Q_{l-2}(t),
\end{split}
\end{equation}
respectively with the initial values 
$P_0(t)=0$, $P_1(t)=1$, $Q_0(t)=1$ and $Q_1(t)=1-t$. 
\end{lemma}

\begin{proof}
First we consider the polynomials $L_l(t)$ together with related polynomials $K_l(t):=e^tD^l\left(e^{-t}t^{l-1}\right)$ 
(both of these are closely connected to the Laguerre polynomials, see \cite{Szego1975}). 
Readily
\begin{equation*}
\begin{split}
K_l(t) &= e^tD^{l-1}D\left(e^{-t}t^{l-1}\right) \\
&= -e^tD^{l-1}\left(e^{-t}t^{l-1}\right) + (l-1)e^tD^{l-1}\left(e^{-t}t^{l-2}\right)\\
& = - L_{l-1}(t) +(l-1) K_{l-1}(t).
\end{split}
\end{equation*}
For $L_l(t)$ we apply again Leibniz's rule:
\begin{align*}
L_l(t) &= e^tD^l\left(e^{-t}t^l\right)\\
&= e^tD^l\left(te^{-t}t^{l-1}\right) \\
&= te^tD^l\left(e^{-t}t^{l-1}\right) + le^tD^{l-1}\left(e^{-t}t^{l-1}\right) \\
&= t K_l(t) + l L_{l-1}(t).
\end{align*}
Consequently,
\begin{equation*}\label{}
L_l(t)  = (2l-1-t) L_{l-1}(t) - (l-1)^2 L_{l-2}(t)
\end{equation*}
with $L_0(t)=1$ and $L_1(t)=1-t$.

The substitution $Q_l(t)=(-t)^l L_l\left(\frac{1}{t}\right)$ confirms \eqref{POLRECURRENCES} for $Q_l(t)$.
Taking linear combinations of the Pad\'e formula \eqref{PADE3} at $l$, $l-1$ and $l-2$  yields
\begin{align*}
  &\left( R_l(t) - (1-(2l-1)t) R_{l-1}(t) + (l-1)^2t^2 R_{l-2}(t) \right) \\
= \; &\left( Q_l(t) - (1-(2l-1)t) Q_{l-1}(t) + (l-1)^2t^2 Q_{l-2}(t) \right) E(t) \\
&- \left( P_l(t) - (1-(2l-1)t) P_{l-1}(t) + (l-1)^2t^2 P_{l-2}(t) \right) \\
= \; &- \left( P_l(t) - (1-(2l-1)t) P_{l-1}(t) + (l-1)^2t^2 P_{l-2}(t) \right),
\end{align*}
where 
$$
\underset{t=0}{\ord}\, \left( R_l(t) - (1-(2l-1)t) R_{l-1}(t) + (l-1)^2t^2 R_{l-2}(t) \right) \ge 2l-2,
$$
$$
\deg \left( P_l(t) - (1-(2l-1)t) P_{l-1}(t) + (l-1)^2t^2 P_{l-2}(t) \right) \le l-1.
$$
Hence 
$$
R_l(t) - (1-(2l-1)t) R_{l-1}(t) + (l-1)^2t^2 R_{l-2}(t) = 0
$$
and
$$ 
P_l(t) - (1-(2l-1)t) P_{l-1}(t) + (l-1)^2t^2 P_{l-2}(t) = 0,
$$
too.
\end{proof}

\section{Continued fractions}\label{Generalized continued fractions}

\subsection{About generalized continued fractions}

Following Lepp\"al\"a et al.\ \cite{Leppala2017}, we give a couple of basic facts about generalized continued fractions. 
See also \cite{Borwein2014}.

A generalized continued fraction is the expression
\begin{equation}\label{GC}
b_0+\frac{a_1}{b_1+\frac{a_2}{b_2+\cdots}},
\end{equation}
shortly denoted by
\[
b_0+\overset{\infty}{\underset{k=1}{\mathbb{K}}}\frac{a_n}{b_n}=
b_0+\frac{a_1}{b_1} \alaplus \frac{a_2}{b_2} \alaplus \aladots \,.
\]
The convergents of \eqref{GC} are defined by
\[
\frac{A_n}{B_n} := b_0+ \overset{n}{\underset{k=1}{\mathbb{K}}} \frac{a_k}{b_k}=b_0+\frac{a_1}{b_1}\alaplus \frac{a_2}{b_2}
\alaplus \aladots \alaplus \frac{a_n}{b_n} \,,
\]
where the numerators $A_n$ and denominators $B_n$ both satisfy the recurrence formula
\begin{equation}\label{ABRECURRENCES}
C_n = b_n C_{n-1} + a_n C_{n-2}
\end{equation}
for $n \in \mathbb{Z}_{\ge 2}$
with initial values $A_0=b_0$, $B_0=1$, $A_1=b_0b_1+a_1$ and $B_1=b_1$. 
By the value of \eqref{GC} we mean the limit
\[
\tau = \lim_{n\to\infty} \frac{A_n}{B_n}
\]
when it exists.
Using the recurrence formula \eqref{ABRECURRENCES} and induction gives
\[
\frac{A_{k+1}}{B_{k+1}}-\frac{A_k}{B_k}=
\frac{(-1)^k a_1a_2\cdots a_{k+1}}{B_kB_{k+1}}
\]
for all $k \in \mathbb{Z}_{\ge 0}$,
which by telescoping implies
\[
\frac{A_n}{B_n}=b_0+\sum_{k=0}^{n-1}\frac{(-1)^ka_1a_2\cdots a_{k+1}}{B_kB_{k+1}}\,.
\]
Supposing the convergence, the limit
\[
\tau=b_0+\sum_{k=0}^{\infty}\frac{(-1)^ka_1a_2\cdots a_{k+1}}{B_kB_{k+1}}
\]
exists, which further implies
\[
\tau-\frac{A_n}{B_n}=\sum_{k=n}^{\infty}\frac{(-1)^ka_1a_2\cdots a_{k+1}}{B_kB_{k+1}}\,.
\]

In the following, we consider the values of continued fractions, where $a_k,b_k\in\mathbb{Q}[t]$ or $a_k,b_k\in\mathbb{Q}$. 
In the first case, the value of the (infinite) continued fraction will be a formal Laurent series in the field $\mathbb{Q}((t))$, 
and in the second case, the value will be in $\mathbb{C}$ or $\mathbb{C}_p$, where $p$ is a prime.

\subsection{A continued fraction for Euler's function and Hardy integral}\label{Acontinuedfraction}

Now we are ready to compute the value of the continued fraction \eqref{contfrac} in the field $\mathbb{Q}((t))$ of formal Laurent series, in the field of $p$-adic numbers and in the field of real numbers, respectively.
Recall the definition of the continued fraction \eqref{contfrac}:
\begin{equation*}\label{}
\frac{1}{1-t-\frac{t^2}{1-3t-\frac{2^2t^2}{1-5t+\ldots} } } := \overset{\infty}{\underset{k=1}{\mathbb{K}}} \left(\frac{a_k}{b_k}\right),
\end{equation*}
where $a_1=1$, $b_1=1-t$, $a_k=(k-1)^2t^2$ and $b_k=1-(2k-1)t$ for $k\in\mathbb{Z}_{\ge 2}$.

\begin{lemma}\label{}
In the field $\mathbb{Q}((t))$ of formal Laurent series,
\begin{equation*}\label{}
\frac{1}{1-t-\frac{t^2}{1-3t-\frac{2^2t^2}{1-5t+\ldots} } } = \sum_{k=0}^{\infty}k!t^k.
\end{equation*}
\end{lemma}

\begin{proof}
The Pad\'e approximations
\begin{equation}\label{PADE5}
E(t) - \frac{P_{l}(t)}{Q_{l}(t) } =\frac{R_{l}(t)}{Q_{l}(t) } 
                                  =l!^2 t^{2l}\,\frac{\sum_{k=0}^{\infty} k!\binom{l+k}{k}^2t^k}{\sum_{h=0}^{l}h!\binom{l}{h}^2(-t)^h}
\end{equation}
from Lemma \ref{EULERPADE} show that
\begin{equation*}\label{order}
\underset{t=0}{\ord} \left( E(t) - \frac{P_{l}(t)}{Q_{l}(t)} \right) = 2l.
\end{equation*}
Thus, as formal Laurent series
\begin{equation*}\label{}
\lim_{l\to\infty} \frac{P_{l}(t)}{Q_{l}(t)} = E(t).
\end{equation*}

On the other hand, it happens that $\frac{P_{l}(t)}{Q_{l}(t)}$ is the $l$th convergent of the continued fraction
\begin{equation}\label{ContHappens}
\frac{1}{1-t-\frac{t^2}{1-3t-\frac{2^2t^2}{1-5t+\ldots} } }.
\end{equation}
Namely, according to Lemma \ref{POLREC}, the Pad\'e polynomials $P_l(t)$ and  $Q_l(t)$ satisfy the recurrences
$$
P_l(t) = (1-(2l-1)t) P_{l-1}(t) - (l-1)^2t^2 P_{l-2}(t),
$$
$$
Q_l(t) = (1-(2l-1)t) Q_{l-1}(t) - (l-1)^2t^2 Q_{l-2}(t)
$$
with the initial values 
$P_0(t)=0$, $P_1(t)=1$, $Q_0(t)=1$ and $Q_1(t)=1-t$.
From \eqref{ContHappens} we see that
$A_0=0$, $A_1=1$, $B_0=1$, and $B_1(t)=1-t$, corresponding to the above initial values.
Hence, by comparison to \eqref{ABRECURRENCES}, $P_l(t)=A_{l}(t)$ and $Q_l(t)=B_{l}(t)$ for all $l\in\mathbb{Z}_{\ge 0}$, implying
\begin{equation*}\label{}
\lim_{l\to\infty} \frac{P_{l}(t)}{Q_{l}(t)} =\lim_{l\to\infty}\frac{A_{l}(t)}{B_{l}(t)} =
 \frac{1}{1-t-\frac{t^2}{1-3t-\frac{2^2t^2}{1-5t+\ldots} } }.
\end{equation*}
\end{proof}

\paragraph{\textbf{Proof of Theorem \ref{CFpadic+real}.}} 
First we prove the $p$-adic evaluation \eqref{pcontinued}:
\begin{proof}[Proof of Theorem \ref{CFpadic+real}, $p$-adic case]
Fix now $t\in\mathbb{C}_p$, $|t|_p<1$. 
Let $\left(\frac{A_l(t)}{B_l(t)}\right)$ be the sequence of convergents  of the continued fraction in \eqref{pcontinued}.
The assumption $|t|_p<1$ implies
\begin{equation*}\label{}
|Q_l(t)|_p=\left| \sum_{h=0}^{l}h!\binom{l}{h}^2(-t)^h \right|_p = 1, \quad 
\left| \frac{R_{l}(t)}{l!^2 t^{2l}}\right|_p =\left| \sum_{k=0}^{\infty} k!\binom{l+k}{k}^2t^k\right|_p = 1.
\end{equation*}
The facts $P_l(t)=A_{l}(t)$, $Q_l(t)=B_{l}(t)$ and \eqref{PADE5} then confirm the limit
\begin{equation}\label{ketjusuppeneeeuler}
\begin{split}
\left| E(t) - \frac{A_{l}(t)}{B_{l}(t)} \right|_p & = \left| E(t) - \frac{P_{l}(t)}{Q_{l}(t)} \right|_p 
= \left| \frac{R_{l}(t)}{Q_{l}(t)} \right|_p \\
&= \left| l!^2 t^{2l}\,\frac{\sum_{k=0}^{\infty} k!\binom{l+k}{k}^2t^k}{\sum_{h=0}^{l}h!\binom{l}{h}^2(-t)^h} \right|_p
= \left| l!^2 t^{2l} \right|_p\ \overset{l \to \infty}{\to} 0.
\end{split}							
\end{equation}
\end{proof}


Next, we prove the Archimedean  evaluation \eqref{rcontinued}.
Let $t\in\mathbb{R}$, $t\le 0$. Then, in the field of real numbers,
\begin{equation*}\label{}
\frac{1}{1-t-\frac{t^2}{1-3t-\frac{2^2t^2}{1-5t+\ldots} } } = \int_{0}^{\infty} \frac{e^{-s}}{1-ts}\d s.
\end{equation*}
\begin{proof}[Proof of Theorem \ref{CFpadic+real}, real case]
Let us apply identity \eqref{Hardyapproximation},
\begin{equation*}\label{}
Q_l(t) \mathcal{H}(t) - P_l(t) = \mathcal{R}_l(t), 
\end{equation*}
from the proof of Lemma \ref{HARDYAPPROX}. 
Write $x=-t>0$; then
\begin{equation*}\label{}
|Q_l(t)|= \sum_{h=0}^{l} h! \binom{l}{h}^2 x^h \ge l!x^l,
\end{equation*}
and, as in \eqref{remainderestimate},
\begin{equation*}\label{}
|\mathcal{R}_l(t)| = l! x^{2l} \int_{0}^{\infty} \frac{e^{-s}s^l}{(1+xs)^{l+1}}\d s \le l! x^{l} \frac{l^l}{(l+1)^{l+1}}. 
\end{equation*}
Analogously to \eqref{ketjusuppeneeeuler}, we arrive at the limit
\begin{equation*}\label{}
\left| \mathcal{H}(t) - \frac{A_{l}(t)}{B_{l}(t)} \right| 
= \left| \mathcal{H}(t) - \frac{P_{l}(t)}{Q_{l}(t)} \right| 
= \left| \frac{\mathcal{R}_l(t)}{Q_{l}(t)} \right|
\le \frac{l^l}{(l+1)^{l+1}} \rightarrow\ 0.							
\end{equation*}
(Just remember that $\mathcal{H}(t)= \int_{0}^{\infty} \frac{e^{-s}}{1-ts}\d s$.)
\end{proof}


\section{Determinant}

\begin{lemma}\label{detlemma}
The determinant $\begin{vmatrix}
Q_l(t) & P_l(t) \\
Q_{l+1}(t) & P_{l+1}(t) 
\end{vmatrix}$
is non-zero.
\end{lemma}

\begin{proof}
We follow \cite{TAWA2018}: From the identity
$$
\begin{vmatrix}
Q_l(t) & P_l(t) \\
Q_{l+1}(t) & P_{l+1}(t) 
\end{vmatrix}
= -
\begin{vmatrix}
Q_l(t) & R_l(t) \\
Q_{l+1}(t) & R_{l+1}(t) 
\end{vmatrix}
$$
based on \eqref{PADE3}, we get
\begin{equation}\label{detidentity}
Q_l(t) P_{l+1}(t) - Q_{l+1}(t) P_l(t) = Q_{l+1}(t) R_l(t) - Q_l(t) R_{l+1}(t).
\end{equation}
Because the degree of the left-hand side is at most $2l$, and the order of the right-hand side is at least $2l$, \eqref{detidentity} must be equal to a constant times $t^{2l}$. Furthermore, from the right-hand side of \eqref{detidentity}, we can deduce that this constant is $l!^2$.
\end{proof}

\section{Proofs of the results}

\subsection{Proof of Theorem \ref{theoremX}}

Denoting
$$
q_l := Q_l (t), \quad p_l := P_l (t), \quad r_l := R_l (t),
$$
Lemma \ref{EULERPADE} gives
$$
q_l E(t) - p_l = r_l.
$$
Then our linear form \eqref{linform} becomes
$$
q_l \Lambda_p = c q_l E_p \left(t\right) - q_l d = c p_l - d q_l + c r_l.
$$
Denote
$$
W_l:= c p_l - d q_l = q_l \Lambda_p - c r_l,
$$
where $c p_l - d q_l \in \Z$. From Lemma \ref{detlemma} it follows that $W_{\hat l} \neq 0$ for $\hat l=l$ or $\hat l=l+1$. 
Hence, for such $\hat l$, the product formula implies

\begin{align*}
1 = \left| W_{\hat l} \right| \prod_{q \in \mathbb{P}} \left| W_{\hat l} \right|_q 
= \left| W_{\hat l}\right| \left| q_{\hat l} \Lambda_p - c r_{\hat l}\right|_p \prod_{q \in \mathbb{P}, q\ne p} \left| W_{\hat l} \right|_q 
\le \left| W_{\hat l} \right| \max\left\{ \left| \Lambda_p \right|_p, \left| r_{\hat l} \right|_p \right\}.
\end{align*}

From Lemmas \ref{Ppolraja2}, \ref{polrajaQ}, and \ref{jaannoskoko}, we get
$$
|p_{\hat l}| \le 2 (l+1)! B(l+1,t)  |t|^{l+1} e^{2\sqrt{\frac{l+1}{|t|}}}, \quad |q_{\hat l}| \le (l+1)!  |t|^{l+1} e^{2\sqrt{\frac{l+1}{|t|}}}, 
\quad |r_{\hat l}|_p = |l!|_p^2 |t|_p^{2l}.
$$
So
\begin{equation*}\label{1lessthan}
\begin{split}
1 &\le \; \left| c p_{\hat l} - d q_{\hat l} \right| \cdot \max\left\{ \left| \Lambda_p \right|_p, \left| r_{\hat l} \right|_p \right\} \\
  &\le \;  (|c|+|d|) \cdot 2 (l+1)! B(l+1,t)  |t|^{l+1} e^{2\sqrt{\frac{l+1}{|t|}}} \cdot  
	         \max\left\{ \left| \Lambda_p \right|_p, |l!|_p^2 |t|_p^{2l}\right\} \\
	           &\le \;  H \cdot 2 (l+1)! B(l+1,t)  |t|^{l+1} e^{2\sqrt{\frac{l+1}{|t|}}} \cdot  
	         \max\left\{ \left| \Lambda_p \right|_p, |t|_p^{2l}\right\}.
\end{split}
\end{equation*}
Since by \eqref{paehto2} $H \cdot 2 (l+1)! B(l+1,t)  |t|^{l+1} e^{2\sqrt{\frac{l+1}{|t|}}} |t|_p^{2l}<1$, we must have $1 \le H \cdot 2 (l+1)! B(l+1,t)  |t|^{l+1} e^{2\sqrt{\frac{l+1}{|t|}}} \cdot | \Lambda_p |_p$ and furthermore
$$
1 \le H \cdot 2 (l+1)! B(l+1,t)  |t|^{l+1} e^{2\sqrt{\frac{l+1}{|t|}}} \cdot | \Lambda_p |_p < |t|_p^{-2l} \cdot |\Lambda_p|_p.
$$
\begin{flushright}
\qed
\end{flushright}

\subsection{Proof of Corollary \ref{theorem1}}

Now $t=\pm p^a$, where $a$ is such that
\begin{equation*}
p^a > \left( H \cdot 2 (l+1)! B(l+1,p^a) e^{2\sqrt{\frac{l+1}{p^a}}} \right)^{\frac{1}{l-1}},
\quad B(l,p^a)= \left( \frac{lp^a}{4} \right)^{\frac{1}{4}}+ \left( \frac{lp^a}{4} \right)^{-\frac{1}{4}}.
\end{equation*}
Therefore,
\begin{equation*}\label{}
H \cdot 2 (l+1)! B(l+1,t) \cdot |t|^{l+1} e^{2\sqrt{\frac{l+1}{|t|}}} \cdot |t|_p^{2l} = H \cdot 2 (l+1)! B(l+1,p^a) e^{2\sqrt{\frac{l+1}{p^a}}} p^{-a(l-1)} < 1,
\end{equation*}
confirming assumption \eqref{paehto2} of Theorem \ref{theoremX}. Hence
\begin{equation*}\label{}
\left| c E_p\!\left(\pm p^a\right)-d \right|_p \ge  \frac{1}{p^{2la}}
\end{equation*}
follows by Theorem \ref{theoremX}.
\begin{flushright}
\qed
\end{flushright}

\subsection{Proof of Corollary \ref{Cordiminish}}

We will use Corollary \ref{theorem1} to prove this corollary. Now we choose $l:=\left\lfloor \frac{\log H}{16} \right\rfloor$. 
It is sufficient to notice that $l\geq 4$. Therefore $B(l+1,\pm p^a)\le \sqrt{2} \ (l+1)^{\frac{1}{4}} |p^a|^{1/4}$.

 Then
\begin{equation}\label{lRestimate}
\frac{\log H}{16} - 1 < l \le \frac{\log H}{16}.
\end{equation}
 We wish to show that
$$
p^a > \left( H \cdot 2\, (l+1)!\, \sqrt{2} (l+1)^{\frac{1}{4}}\, e^{2\sqrt{\frac{l+1}{2}}} \right)^{\frac{1}{l-\frac{5}{4}}} =: \rho (l),
$$
because since
\[
\left( H \cdot 2\, (l+1)!\, \sqrt{2} (l+1)^{\frac{1}{4}}\, e^{2\sqrt{\frac{l+1}{2}}} \right)^{\frac{1}{l-\frac{5}{4}}}\geq 
\left( H \cdot 2 (l+1)! B(l+1,p^a) e^{2\sqrt{\frac{l+1}{p^a}}} \right)^{\frac{1}{l-1}},
\]
then condition \eqref{paehto} of Corollary \ref{theorem1} follows.

We use the fact that $n!\le (n-1)^{n-1}$ for $n\ge 4$ and the bound $l-\frac{5}{4}\geq \frac{11}{16}l$ . We obtain
\begin{align*}
\log (\rho (l)) &=\frac{\log(2H)+\log((l+1)!)+\frac{1}{4}\log(l+1)+\log(\sqrt{2})+2\sqrt{\frac{l+1}{2}}}{l-\frac{5}{4}} \\
&\leq \frac{\log (2H)}{\frac{11}{16}l}+\frac{l\log l}{l-\frac{5}{4}}+\frac{2\sqrt{\frac{l+1}{2}}+\frac{1}{4}\log(l+1)+
\log \sqrt{2}}{\frac{11}{16}l}.
\end{align*}

We have
\[
\frac{l\log l}{l-\frac{5}{4}}-\log l=\frac{\frac{5}{4}\log l}{l-\frac{5}{4}}\leq \frac{20\log l}{11l} \leq \frac{5\log 4}{11}
\]
and
\[
\frac{2\sqrt{\frac{l+1}{2}}+\frac{1}{4}\log(l+1)+\log \sqrt{2}}{\frac{11}{16}l}\leq \frac{2\sqrt{\frac{4+1}{2}}+\frac{1}{4}\log(4+1)+
\log \sqrt{2}}{\frac{11}{16}\cdot 4}\leq 1.42226
\]
for $l\geq 4$.

Hence
\begin{align*}
\log (\rho (l)) &\leq \frac{\log (2H)}{\frac{11}{16}l}+\log l+ \frac{5\log 4}{11} +1.42226\\ & \leq \frac{\log H}{\frac{11}{16}l}+
\log\left(\frac{\log H}{16}\right)+\frac{\log 2}{\frac{11}{16}l}+ \frac{5\log 4}{11} +1.42226\\ 
& \leq \frac{\log H}{\frac{11}{16}\left(\frac{\log H}{16}-1\right)}+\log\log H+\frac{4\log 2}{11}+ \frac{5\log 4}{11} +1.42226-\log 16\\ 
& =\frac{16^2}{11}+\frac{16^3}{11(\log H-16)}+\log\log H+\frac{4\log 2}{11}+ \frac{5\log 4}{11} +1.42226-\log 16\\ 
& \leq \log\log H+22.9302\leq 4\log\log H
\end{align*}
since $\log\log  H\geq 8$. This proves the bound 
$p^a > \left( H \cdot 2\, (l+1)!\, \sqrt{2} (l+1)^{\frac{1}{4}}\, e^{2\sqrt{\frac{l+1}{2}}} \right)^{\frac{1}{l-\frac{5}{4}}} $ 
because $a$ was assumed to satisfy the bound $p^a>(\log H)^4$.


By Corollary \ref{theorem1},  
\begin{equation*}\label{}
\left|c E_p(p^a)-d \right|_p \ge \frac{1}{p^{2al}} \ge \frac{1}{ p^{\frac{a}{8}\log H} }.
\end{equation*}
\begin{flushright}
\qed  
\end{flushright}

\subsection{Proof of Corollary \ref{Cormin}}

Again we have $t=\pm p^a$ and thus $|t|\ge 2$. 
Let us assume $l\ge 2$; then $B(l+1,\pm p^a)\le \sqrt{2} \ (l+1)^{\frac{1}{4}} |p^a|^{1/4}$. Again, we wish to show that
\begin{equation}\label{assumption2}
p^a > \left( H \cdot 2\, (l+1)!\, \sqrt{2} (l+1)^{\frac{1}{4}}\, e^{2\sqrt{\frac{l+1}{2}}} \right)^{\frac{1}{l-\frac{5}{4}}} =: \rho (l).
\end{equation}
In this proof we are going to first bound upwards the value of $\rho(l)$ and then choose such a value of $l$ that this upper bound is as small as possible. The choice we will make will be $l= \left\lceil \frac{16}{11}\log(2H) \right\rceil$, which is at least 4 due to the assumption $H \ge 4$. Therefore, throughout the proof, we may actually assume that $l \ge 4$.

We estimate in a fairly similar manner as in the proof of the previous corollary:
\begin{equation}\label{logrho}
\begin{split}
\log \rho (l) &\le \frac{1}{l-\frac{5}{4}} \left( \log(2H) + l \log l + \log \sqrt{2} + \frac{1}{4} \log (l+1) + 2\sqrt{\frac{l+1}{2}} \right) \\
&\le \frac{\log(2H)}{l-\frac{5}{4}} + \log l + \frac{5}{4} \cdot \frac{\log l}{l - \frac{5}{4}} + \frac{\log \sqrt{2}}{4-\frac{5}{4}} + \frac{\log 5}{4 \cdot 4 -5} + \frac{2 \sqrt{\frac{5}{2}}}{4 - \frac{5}{4}} \\
&\le \frac{\log(2H)}{\frac{11}{16}l} + \log l + \frac{5 \log 4 + 4\log \sqrt{2} + \log 5 + 8 \sqrt{\frac{5}{2}}}{11}, \\
\end{split}
\end{equation}
where the right-hand side attains its minimum at $l=\frac{16}{11}\log(2H)$. Hence we choose 
$l= \left\lceil \frac{16}{11}\log(2H) \right\rceil$.
Note that
$$
\frac{16}{11}\log(2H) \le l < \frac{16}{11}\log(2H) + 1 = \frac{16}{11} \log \left( 2H e^{\frac{11}{16}} \right).
$$
It follows that
\begin{equation*}\label{}
\log\rho(l) \le 1 + \log \log \left(2H e^{\frac{11}{16}} \right) + \log \frac{16}{11} + \frac{5 \log 4 + 4\log \sqrt{2} + \log 5 + 8 \sqrt{\frac{5}{2}}}{11} < a\log p
\end{equation*}
by assumption \eqref{paehto4}. Hence \eqref{assumption2} holds for $l = \left\lceil \frac{16}{11}\log(2H) \right\rceil$.
Therefore, by Corollary \ref{theorem1}, we obtain
\begin{equation*}\label{}
\left| \Lambda_p \right|_p \ge \frac{1}{p^{2al}} > p^{-2a \cdot \frac{16}{11} \log \left( 2H e^{\frac{11}{16}} \right)} = \frac{1}{\left( 2H e^{\frac{11}{16}} \right)^{\frac{32}{11} a \log p}}.
\end{equation*}
\begin{flushright}
\qed
\end{flushright}

\section*{Acknowledgements}

The authors would like to thank the anonymous referees for their very careful reading of the manuscript and detailed comments.

The research of Ernvall-Hyt\"onen and Sepp\"al\"a was supported by the Emil Aaltonen Foundation. 
A big part of the work of Sepp\"al\"a was conducted during her time at Aalto University supported by a grant from the Magnus Ehrnrooth Foundation.

\end{document}